\documentclass[11pt]{amsart}
\usepackage{amsmath}
\usepackage{amssymb}
\usepackage{amsfonts}
\usepackage{latexsym}
\usepackage{color}
\usepackage{graphicx}

\newcommand{\be}{\begin{equation}}
\newcommand{\en}{\end{equation}}
\newcommand{\bea}{\begin{eqnarray}}
\newcommand{\ena}{\end{eqnarray}}
\newcommand{\beano}{\begin{eqnarray*}}
\newcommand{\enano}{\end{eqnarray*}}
\newcommand{\bee}{\begin{enumerate}}
\newcommand{\ene}{\end{enumerate}}
\newcommand{\ad}{^{\mbox{\scriptsize $\dag$}}}

\newcommand{\mc}{\mathcal}
\newcommand{\mb}{\mathbb}
\newcommand{\A}{\mathfrak{A}}
\newcommand{\Ao}{\mathfrak{A}_0}

\newcommand{\Hil}{{\mc H}}

\newcommand{\F}{{\mathbb F}}

\def\L{{\mathcal L}}
\newcommand{\Lc}{{\mc L}}

\newcommand{\D}{{\mc D}}

\newcommand{\M}{{\mathfrak M}}

\newcommand{\BB}{{\mathfrak B}}

\newcommand{\up}{\upharpoonright}

\newtheorem{defn}{Definition}[section]
\newtheorem{thm}[defn]{Theorem}
\newtheorem{prop}[defn]{Proposition}
\newtheorem{lemma}[defn]{Lemma}

\newtheorem{example}[defn]{Example}
\newtheorem{rem}[defn]{Remark}

\def\x{\relax\ifmmode {\mbox{*}}\else*\fi}
\newcommand{\beex}{\begin{example}$\!\!${\bf }$\;$\rm }
\newcommand{\enex}{ \end{example}}
\newcommand{\berem}{\begin{rem}$\!\!${\bf }$\;$\rm }
\newcommand{\enrem}{ \end{rem}}
\newcommand{\bedefi}{\begin{defn}$\!\!${\bf }$\;$\rm }
\newcommand{\findefi}{\end{defn}}
\newcommand{\pa}{partial \mbox{*-algebra}}

\newcommand{\ha}{^{\ast}}

\newcommand{\ip}[2]{\left\langle {#1}\left|{#2}\right.\right\rangle}
\newcommand{\gl}{{\mathfrak L}}
\newcommand{\LDD}{\gl(\D,\D^\times)}

\newcommand{\LD}{\Lc^\dagger(\D)}

\newcommand{\LBDD}[1]{\mbox{${\gl}_{\textsf{B}}^{#1}(\D,\D^\times)$}}
\newcommand{\LBDDs}{\mbox{${\sf L}_{\textsf{B}}(\D,\D^\times)$}}
\newcommand{\LBDE}{\LBDD{}}

\newcommand{\LGD}{\gl^\dagger(\D)}

\newcommand{\itop}{\tau_{\rm ind}}
\def\H{{\mathcal H}}

\newcommand{\restr}[1]{_{\!\up{#1}}}

\newcommand{\m}{{\scriptscriptstyle\mathfrak M\;}}

\begin{document}
\title[Bounded elements of C*-inductive spaces]
{Bounded elements of C*-inductive locally convex spaces}

\author{Giorgia Bellomonte}
\author{Salvatore di Bella}
\author{Camillo Trapani}
\address{Dipartimento di Matematica e Informatica,
Universit\`a di Palermo, I-90123 Palermo, Italy}
\email{bellomonte@math.unipa.it}
\email{salvatore.dibella@math.unipa.it}
\email{camillo.trapani@unipa.it}


\begin{abstract} The notion of bounded element of C*-inductive locally convex spaces (or C*-inductive partial *-algebras) is introduced and discussed in two ways: the first one takes into account the inductive structure provided by certain families of C*-algebras; the second one is linked to natural order of these spaces. A particular attention is devoted to the relevant instance provided by the space of continuous linear maps acting in a rigged Hilbert space.
\end{abstract}

\maketitle

\section{Introduction}\label{sect_introd}
Some locally convex spaces exhibit an interesting feature: they contain a large number of C*-algebras that often contribute to their topological structure, in the sense that these spaces can be thought as {\em generalized}  inductive limits of C*-algebras. These objects were called {\em C*-inductive locally convex spaces} in \cite{betra2} and their structure was examined in detail, also taking in mind that they arise naturally when one considers the operators acting in the {\em joint topological limit} of an inductive family of Hilbert spaces as described in \cite{betra}. Indeed, a typical instance of this structure is obtained by considering the space $\LBDD{}$ of operators acting in the rigged Hilbert space canonically associated to an O*-algebra of unbounded operators acting on a dense domain $\D$ of Hilbert space $\H$. In \cite{betra2} a series of features of this structure were studied giving a particular attention to the order structure, positive linear functionals and representation theory. The space $\LBDD{}$ contains a subspace isomorphic to the *-algebra $\BB(\H)$ of bounded operators in $\H$ whose elements can be in natural way considered as the {\em bounded elements} of $\LBDD{}$.
The notion of bounded element of a locally convex *-algebra  $\A$ was first introduced by Allan \cite{allan} with the aim of developing a spectral theory for topological *-algebras: an element $x$ of the topological *-algebra  $\A[\tau]$ is {\em Allan bounded} if there exists $\lambda \neq 0$ such that the set $\{(\lambda^{-1}x)^n;\, n=1,2,\ldots\}$ is a bounded subset of $\A[\tau]$. This definition was suggested by the successful spectral analysis for closed operators in Hilbert space $\H$: a complex number $\lambda$ is in the resolvent set $\rho(T)$ of a closed operator $T$ if $T-\lambda I$ has an inverse in the *-algebra $\BB(\H)$ of bounded operators.

There are, however, several other possibilities for defining bounded elements. For instance, one may say that $x$ is bounded if $\pi(x)$ is a bounded operator, for every (continuous, in a certain sense) *-representation $\pi$ defined on a dense domain $\D_\pi$ of some Hilbert space $\H_\pi$.   This could be a reasonable definition in itself, provided that  $\A$ possesses sufficiently many *-representations in Hilbert space.

Moreover some attempts to extend this notion to the larger set-up of locally convex quasi *-algebras \cite{ctbound1, ctbound2, ct_ban, frag_ct_st} or locally convex partial *-algebras \cite{antratsc, att_2010, ant_be_ct} have been done. But in these cases, Allan's notion cannot be adopted, since powers of a given element $x$ need not be defined.

In the case of *-algebras, bounded elements in purely algebraic terms have been considered by Vidav \cite{vidav} and Schm\"udgen \cite{schm_weyl} with respect to some (positive) wedge.

The aim of this paper is to extend the notion of bounded element to the case of C*-inductive locally convex spaces $\A$ with defining family of C*-algebras $\{\BB_\alpha; \alpha \in \F\}$ ($\F$ is an index set directed upward). There are also in this case several possibilities: the first one consists in taking elements that have {\em representatives} in every C*-algebra $\BB_\alpha$ of the family whose norms are uniformly bounded; the second one consists into taking into account the order structure of $\A$, in the same spirit of the quoted papers of Vidav and Schm\"udgen.

The paper is organized as follows. After some preliminaries (Section \ref{sect_preliminaries}), we study, in Section \ref{sect_oprhs}, how {\em bounded elements} of $\LBDD{}$ can be derived from its C*-inductive structure and from its order structure. We show that these two notions are equivalent and that an element $X$ is bounded if and only if $X$ maps $\D$ into $\H$ and $\overline{X}\in \BB(\H)$.
Finally, in Section \ref{sect_abstract}, we consider the same problem for abstract C*-inductive locally convex spaces and give conditions for some of the characterizations proved for $\LBDD{}$ maintain their validity.
Some of these results are then specialized to the case where $\A$ is a C*-inductive locally convex partial *-algebra.

\section{Notations and preliminaries}\label{sect_preliminaries}

For general aspects of the theory of \pa s and of their representations, we refer to the monograph
\cite{ait_book}.
For the convenience of the reader, however, we   repeat here the essential definitions.

A \pa\ $\A$ is a complex vector space with conjugate linear
involution  $\ha $ and a distributive partial multiplication
$\cdot$, defined on a subset $\Gamma \subset \A \times \A$,
satisfying the property that $(x,y)\in \Gamma$ if, and only if,
$(y\ha ,x\ha )\in   \Gamma$ and $(x\cdot y)\ha = y\ha \cdot x\ha $.
From now on, we will write simply $xy$ instead of $x\cdot y$ whenever
$(x,y)\in \Gamma$. For every $y \in \A$, the set of left (resp.
right) multipliers of $y$ is denoted by $L(y)$ (resp. $R(y)$), i.e.,
$L(y)=\{x\in \A:\, (x,y)\in \Gamma\}$, (resp. $R(y)=\{x\in \A:\, (y,x)\in \Gamma\}$). We denote by $L\A$ (resp.
$R\A$)  the space of universal left (resp. right) multipliers of
$\A$.
In general, a \pa\ is not associative.

The {\em unit} of  partial *-algebra $\A$, if any, is an element $e\in \A$ such that $e=e\ha$, $e\in R\A\cap L\A$ and $xe=ex=x$, for every $x\in \A$.

Let $\H$ be a complex Hilbert space and $\D$ a dense subspace of $\H$.
 We denote by $ \L\ad(\D,\H) $
the set of all (closable) linear operators $X$ such that $ {D}(X) = {\D},\; {D}(X\x) \supseteq {\D}.$ The map $X\to X^\dagger=X^*\restr\D$ defines an involution on $\L\ad(\D,\H )$, which can be made into a  \pa\ with respect to the {\em weak} multiplication \cite{ait_book}; however, this fact will not be used in this paper.

Let $\Lc^\dagger(\D)$ be the subspace of $\Lc^\dagger(\D,\Hil)$ consisting of all its elements  which leave, together with their adjoints, the domain $\D$ invariant. Then $\Lc^\dagger(\D)$ is a *-algebra with
respect to the usual operations. A *-subalgebra $\M$ of $\Lc^\dagger(\D)$ is called an O*-algebra.

Let $\M $ be an O*-algebra. The {\em graph topology} $t_\m $ on $\D$ is the locally convex topology
defined by the family
$\{\|\cdot\|_A\}_{A\in \M}$, where
$$ \|\xi\|_A= \sqrt{\|\xi\|^2 + \|A\xi\|^2}=\|(I+A^*\overline{A})^{1/2}\xi\|, \quad \xi \in \D.$$
For $A=0$, the null operator of $\LD$, $\|\cdot\|_0$ is exactly the norm of $\H$, thus we will omit the $0$ in the notation of the norm.

The topology
$t_{\m}$ is finer than the norm topology, unless ${\M}$ does  consist of
bounded operators only.

If $\M$ is an O*-algebra, we write $A\preceq B$ if $\|A\xi\|\leq \|B\xi\|$, for every $\xi \in \D$. Then, $\M$ is directed upward with respect to this order relation.

If $A \in \M$, we denote by $\H_A$ the Hilbert space obtained by endowing $D(\overline{A})$ with the graph norm $\|\cdot\|_A$.

If $A,B\in \M$ and $A\preceq B$, then $U_{BA}= (I+B^*\overline{B})^{-1/2}(I+A^*\overline{A})^{1/2}$ is a contractive map of $\H_A$ into $\H_B$; i.e.,
$\|U_{BA}\xi \|_B\leq \|\xi\|_A$, for every $\xi \in \H_A$.

If the
locally convex space $\D[t_\m ]$ is complete, then $\M $ is said to be {\em closed}.

If $\M=\Lc^\dagger(\D)$ then the corresponding graph topology is denoted by $t_\dagger$ instead of $t_{\Lc^\dagger(\D)}$.

\medskip
As is known, a locally convex topology $t$ on $\D$ finer than the topology induced by the Hilbert norm defines, in standard fashion,
a {\em rigged Hilbert space} (RHS)
$$
\D[t] \hookrightarrow  \H \hookrightarrow\D^\times[t^\times],
$$
where $\D^\times$  is the vector space of all continuous conjugate linear functionals on
$\D[t]$, i.e., the conjugate dual of $\D[t]$,
endowed with the {\em strong dual topology} $t^\times= \beta(\D^\times,\D)$ and $\hookrightarrow $ denotes a continuous embedding with dense range. The
Hilbert space $\H$ is identified (by considering the form which puts $\D$ and $\D^\times$ as an extension of the inner product of $\D$) with a dense subspace of
$\D^\times[t^\times]$.

\medskip Let $\LDD$ denote the vector space of all continuous linear maps from $\D[t]$ into  $\D^\times[t^\times]$. In $\LDD$ an involution $X \mapsto X\ad$ can be introduced by the equality
$$ \ip{X\xi}{\eta}= \overline{\ip{X\ad \eta}{ \xi}}, \quad \forall \xi, \eta \in \D.$$  Hence $\LDD$ is a *-invariant vector space.

  To every $X \in \LDD$ there corresponds a separately continuous sesquilinear form $\theta_X$ on $\D \times \D$ defined by
$$
\theta_X (\xi,\eta)= \ip{X\xi}{\eta}, \quad \xi ,\, \eta \in \D.
$$
 The vector space of all {\em jointly} continuous sesquilinear forms on $\D \times \D$ will be
 denoted with ${\sf B}(\D,\D)$.
We denote by $\LBDD{}$ the subspace of  all $X \in \LDD$ such that $\theta_X \in {\sf B}(\D,\D)$ and by $\LGD$ the *-algebra consisting of all operators of $\Lc^\dagger(\D)$, which together with their adjoints are continuous from $\D[t]$ into $\D[t]$. If $t = t_\dagger$, then $\LGD=\Lc^\dagger(\D)$.
We will refer to the rigged Hilbert space defined by endowing $\D$ with the topology $t_\dag$ as to the {\em canonical} rigged Hilbert space defined by $\LD$ on $\D$. In this case $(\LBDD{}, \LD)$ is a quasi *-algebra \cite{ait_book}.

The spaces $\LDD$ and $\LBDD{}$ have been studied at length by several authors (see, e.g. \cite{kurst1,kurst2,kurst3, tratschi}) and several pathologies concerning their multiplicative structure have been considered (see also \cite{ait_book, jpact_book} and references therein). Recently some spectral properties of operators of these classes have also been studied \cite{bedbct}.

\section{Bounded elements of $\LBDD{}$} \label{sect_oprhs}
The inductive structure of $\LBDD{}$, with $\D$ endowed with the graph topology $t_\dagger$, has been discussed in \cite[Section 5]{betra2}. To keep the paper reasonably self-contained, we sum the main features up. 

By the definition itself, $X\in \LBDD{}$  if, and only if, there exists
$\gamma_X>0$ and $A\in \LD$ such that
\begin{equation}\label{beta2} |\theta_X(\xi, \eta)| = |\ip{X\xi}{\eta}| \leq \gamma_X \|\xi\|_A\, \|\eta\|_A, \quad  \forall \xi, \eta \in \D.
\end{equation}

Conversely, if $\theta \in \textsf{B}(\D,\D)$, there exists a unique $X \in \LBDD{}$ such that $\theta= \theta_X$.

Thus, the map $${\mathbb I}: X \in \LBDD{}\mapsto \theta_X \in \textsf{B}(\D,\D)$$ is an isomorphism of vector spaces and ${\mathbb I}(\theta^*)=X^\dag$, where $\theta^*(\xi,\eta)=\overline{ \theta(\eta,\xi)}$, for every $\xi,\eta\in \D$.

We denote by $\textsf{B}^A(\D,\D)$  the subspace of
$\textsf{B}(\D,\D)$ consisting of all $\theta \in \textsf{B}(\D,\D)$
such that \eqref{beta2} holds for fixed $A \in \LD$.

If $\theta \in \textsf{B}^A(\D,\D)$,
it extends to a bounded sesquilinear form on $\H_A\times \H_A$ (we use the same symbol for this
extension). Hence, there exists a unique operator $X_A^\theta\in {\BB}(\Hil_A)$ such that
$$ \theta(\xi, \eta)=\ip{X_A^\theta\xi}{\eta}_A, \quad \forall \xi, \eta \in \Hil_A.$$
On the other hand, if $X_A \in \BB(\Hil_A)$, then the sesquilinear form
$\theta_{X_A}$ defined by
$$ \theta_{X_A}(\xi, \eta)= \ip{X_A\xi}{ \eta}_A, \quad \xi, \eta \in \D, $$
is an element of $\textsf{B}^A(\D,\D)$  and the map
$$ \Phi_A: X_A \in \BB(\Hil_A) \to \theta_{X_A} \in \textsf{B}^A(\D,\D)$$
is a *-isomorphism  of vector spaces with involution.

If $B\succeq A$, then, for $\xi, \eta \in \D$,
$$| \theta_{X_A}(\xi, \eta)|=|\ip{X_A\xi}{ \eta}_A|\leq \|X_A\|_{A,A} \|\xi\|_A\, \|\eta\|_A\leq \|X_A\|_{A,A} \|\xi\|_B\, \|\eta\|_B,$$
where $\|\cdot\|_{A,A}$ denotes the operator norm in $\BB(\Hil_A)$.
Hence, there exists a unique $X_B \in \BB(\H_B)$ such that
$$ \ip{X_A\xi}{ \eta}_A = \ip{X_B\xi}{ \eta}_B, \quad \forall \xi, \eta \in \D.$$
So it is natural to define
$$J_{BA}(X_A)=X_B, \quad \forall X_A \in \BB(\H_A).$$
It is easily seen that $J_{BA}= \Phi_B^{-1}  \Phi_A$.

The space $\LBDD{A}:= {\mathbb I}^{-1}\textsf{B}^A(\D,\D)$ is
a Banach space, with norm
$$ \|X\|^A := \sup_{\|\xi\|_A, \|\eta\|_A\leq 1}|\theta_X(\xi, \eta)|$$
and $\LBDE$ can be endowed with the inductive topology $\itop$ defined by the family
of subspaces $\{\LBDD{A}; A \in \LD\}$ as in \cite[Section 1.2. III]{schmu}.

In conclusion,
$$ X_A \in \BB(\H_A) \leftrightarrow \theta_{X_A} \in \textsf{B}^A(\D,\D)\leftrightarrow X \in \LBDD{A}$$
are isometric *-isomorphisms of Banach spaces.

Hence, to every $X \in \LBDD{}$ one can associate the net $\{X_B; B \in \LD; B\succeq A\}$ of its representatives in each of the spaces $\H_B$.

\bedefi \label{def_bounded} We say that $X \in \LBDD{}$ is a {\em bounded element} of $\LBDD{}$ if
$X$ has a representative $X_A$ in every $\BB(\mathcal{H}_A)$ and $$\|X\|_b :=\sup_{A\in\mathcal{L}^\dag(\mathcal{D})}\| X_A\| _{A,A}<+\infty.$$
\findefi

The space $\LBDD{}_b$ of all bounded elements of $\LBDD{}$ is a Banach space with norm $\|\cdot\|_b$.

\begin{prop} $\mbox{$\LBDD{}_b$}$ is *-isomorphic (as Banach space) to a C*-algebra of operators.
\end{prop}
\begin{proof}Let $\mathcal{H}_\oplus$ denote the Hilbert space direct sum of the ${\mathcal{H}_A, A\in \mathcal{L}^\dag(\mathcal{D})}$; i.e.,
\begin{align*}
\mathcal{H}_\oplus &:= \bigoplus _{A\in\mathcal{L}^\dag(\mathcal{D})} \mathcal{H}_A \\&= \left\{\xi_\oplus =(\xi_A) ; \xi_A\in\mathcal{H}_A, \forall A \in\mathcal{L}^\dag(\mathcal{D})\mbox{ and }\sum_A \| \xi_A\| _A ^2 <+\infty\right\}.
\end{align*}
 If $\{X_A\}_{A\in\mathcal{L}^\dag(\mathcal{D})}$ is a net of operators $X_A\in\BB(\mathcal{H}_A)$, $A\in\mathcal{L}^\dag(\mathcal{D})$, we define $X_\oplus \xi_\oplus = \{X_A \xi_A\}$ provided that $\sum_A \| X_A \xi_A \|^2<+\infty$, $\xi_A\in \H_A$.

The operator $X_\oplus = \{X_A\}$ is bounded if and only if $\sup_A \|X_A\| _{A,A}<+\infty$. The space constructed in this way is $\prod_A \BB(\mathcal{H}_A) = \BB(\mathcal{H}_\oplus)$.
To every $X \in \LBDD{}_b $ we can associate the net $\{X_A\}$ which we have defined above. Clearly, $\{X_A\} \in \BB(\mathcal{H}_\oplus)$. It is easily seen that the map
$$\tau: X \in \LBDD{}_b \mapsto \{X_A\} \in \BB(\mathcal{H}_\oplus)$$
is isometric. Thus, the statement is proved.
\end{proof}

\berem
An element $X\in \LBDD{}$ having a representative $X_A$ for every $A\in\mathcal{L}^\dag(\mathcal{D})$ need not be bounded in the sense of Definition \ref{def_bounded}. The spaces $\{\H_A; A\in \LD\}$, together with their conjugate duals make $D^\times$ into an indexed {\sc PIP}-space   \cite[Ch.2]{jpact_book}. In that language, operators having representatives in every $\H_A$ are called totally regular operators. For more details on their behavior see \cite[Sect. 3.3.3]{jpact_book} where also a C*-agebra corresponding to our bounded elements has been studied.
\enrem

Our next goal is to characterize bounded elements of $\LBDD{}$ in several different ways. For doing this, we need to consider the natural order structure of $\LBDD{}$.

We say that $X\in\LBDD{}$ is {\em positive}, and write $X\geq 0$, if  $\ip{X\xi}{\xi}\geq 0$, for every $\xi\in\mathcal{D}$.

It is easy to see that,
if $X$ is positive, then it is {\em symmetric}; i.e., $X=X^\dag$.

\begin{prop}
The following conditions are equivalent.
\begin{itemize}
\item[(i)] $X\geq 0$.
\item[(ii)] There exists $A\in \mathcal{L}^\dag(\mathcal{D})$ such that $X_B\geq 0,\quad \forall B\succeq A$.
\end{itemize}
\end{prop}
\begin{proof}
(i)$\Rightarrow$(ii):
Since $X\in\LBDD{}$, there exists $A\in\mathcal{L}^\dag(\mathcal{D})$ and $\gamma > 0$ such that
\begin{equation*}
|\ip{X\xi}{\eta}|\leq\gamma \|\xi\| _B \|\eta\| _B, \quad B\succeq A.
\end{equation*}
If $X\geq 0$, then, for every $\xi \in \D$,
 $$\ip{X_B\xi}{\xi}_B=\ip{X\xi}{\xi}\geq 0, \quad \forall B\succeq A.$$
Since $\D$ is dense in $\mathcal{H}_B$, we have
$\ip{X_B\xi}{\xi}_B\geq 0, \; \forall \xi\in\mathcal{H}_B$.\\
(ii)$\Rightarrow$(i):
Let $X_B\geq 0$ for every $B\succeq A$. Then, for every $\xi \in \D$, $\ip{X\xi}{\xi}=\ip{X_B\xi}{\xi}_B\geq 0$.
\end{proof}

\begin{thm}\label{prop_35}
Let $X\in\LBDD{}$. The following statements are equivalent.
\begin{itemize}
\item[(i)]$X:\D\to \H$ and $\overline{X}\in\mathcal{B}(\H)$.
\item[(ii)] $X \in \LBDD{}_b$.
\item[(iii)] There exists $\lambda>0$ such that
\begin{equation*}
-\lambda I\leq \Re(X)\leq \lambda I,\quad -\lambda I\leq \Im(X)\leq \lambda I
\end{equation*}
where $\Re(X)= \frac{X+X^\dag}{2}$ and  $\Im(X)= \frac{X-X^\dag}{2i}$.
\end{itemize}
\end{thm}
\begin{proof} (i)$\Rightarrow$(ii): If $X:\mathcal{D}\rightarrow\mathcal{H}$ and $X$ is bounded, then, for every $A\in \mathcal{L}^\dag(\mathcal{D})$,
\begin{equation} \label{diseq2}
 |\ip{X\xi}{\eta}|\leq\| \overline{X} \| \| \xi\| \| \eta\| \leq \| \overline{X} \| \| \xi\| _A \| \eta\| _A .\quad
\end{equation}
This means that $X$ has a bounded representative $X_A$ in every $\mathcal{B}(\H_A)$. By (\ref{diseq2}), $\| X_A\| _{A,A}\leq \|\overline{X}\|$, for every $A\in \mathcal{L}^\dag(\mathcal{D}) $, so $\sup_{A\in\mathcal{L}^\dag(\mathcal{D})} \| X_A\| _{A,A}  < +\infty$.

(ii)$\Rightarrow$(i) Let $X \in \LBDD{}_b$. Then, for every $A \in \mathcal{L}^\dag (\mathcal{D})$
\begin{equation*}
|\ip{X\xi}{\eta}|\leq\| X_A \|_{A,A}\, \| \xi\| _A\,\| \eta\| _A, \quad \forall \xi, \eta \in \D.
\end{equation*}
In particular, for $A=0$,\begin{equation} \label{diseq} |\ip{X\xi}{\eta}|\leq\| X_0 \| \| \xi\| \| \eta\|, \quad \forall \xi, \eta \in \D.
\end{equation}

By (\ref{diseq}), for every $\xi\in \D$, $F(\eta)$ $=$ $\ip{X\xi}{\eta}$ is a bounded conjugate linear functional on $\mathcal{D}$, so by Riesz's lemma $X\xi\in\mathcal{H}$. It is, finally easily seen that $\overline{X}\in\mathcal{B}(\H)$.

(iii)$\Rightarrow$(i)
Suppose first that $X=X^\dag$. Note that the operator $X$ satisfies the following: $0\leq\frac{X+\lambda I}{2\lambda}\leq I$; so $\frac{X+\lambda I}{2\lambda}$ is a positive operator and $\ip{\frac{X+\lambda I}{2\lambda}\xi}{\xi}\leq \ip{\xi}{\xi}, \ \forall \xi\in \mathcal{D}$; this implies that
\begin{equation}\label{diseq3}
\left|\ip{\frac{X+\lambda I}{2\lambda}\xi}{ \eta}\right|\leq \|\xi\|\ \|\eta\|, \quad \forall \xi , \eta \in \mathcal{D}
\end{equation}
and by Riesz's lemma there exists $\zeta\in\mathcal{H}$ such that
\begin{equation}
\ip{\frac{X+\lambda I}{2\lambda}\xi}{\eta} = \ip{\zeta}{\eta}, \quad \forall \xi , \eta \in \mathcal{D}
\end{equation}
and then $\frac{X+\lambda I}{2\lambda}\xi\in\mathcal{H}$. This implies that $X\xi\in\mathcal{H}$ too.
Moreover, $X$ has a representative for every $A\in\mathcal{L}^\dag(\mathcal{D})$. Indeed,
\begin{equation*}
|\ip{X\xi}{\eta}|\leq \gamma\|\xi\|\|\eta\|\leq\gamma\|\xi\| _A \|\eta\| _A \quad  \forall A\in\mathcal{L}^\dag(\mathcal{D}),
\end{equation*}
where $\gamma> 0$. From (\ref{diseq3}) it follows that $X$ is bounded and $\overline{X}\in\mathcal{B}(\mathcal{H})$.
In the very same way one can prove the boundedness of $X$ if $X^\dag=-X$. The result for a general $X$ follows easily.

(i)$\Rightarrow$ (iii): this is a standard result of the C*-algebras theory.
\end{proof}

\section{Bounded elements of C*-inductive locally convex spaces}\label{sect_abstract}
The results obtained in Section \ref{sect_oprhs} have an abstract generalization to locally convex spaces that are inductive limits of C*-algebras in a generalized sense. These spaces were called {\em C*-inductive locally convex spaces} in \cite{betra2}. We begin with recalling the basic definitions.

\medskip
Let $\A$ be a vector space over ${\mb C}$. Let $\F$ be a set of
indices directed upward and consider, for every $\alpha\in\F$,
 a Banach space $\A_\alpha \subset \A$ such that:
 \begin{itemize}
\item[(I.1)] $\A_\alpha\subseteq\A_\beta$, if $\alpha\leq\beta$;
\item[(I.2)] $\A=\bigcup_{\alpha\in\F}\A_\alpha$;
\item[(I.3)] $\forall\alpha\in\F$, there exists a C*-algebra $\BB_\alpha$ (with unit $e_\alpha$ and norm $\|\cdot\|_\alpha$)   and a norm-preserving isomorphism of vector spaces $\phi_\alpha:\BB_\alpha\to\A_\alpha$;
\item[(I.4)] $x_\alpha\in\BB_\alpha^+\Rightarrow  x_\beta=(\phi_\beta^{-1} \phi_\alpha) (x_\alpha)\in\BB_\beta^+$, for every $\alpha, \beta \in \F$ with $\beta \geq \alpha$.
\end{itemize}

\medskip
We put $j_{\beta\alpha}= \phi_\beta^{-1} \phi_\alpha $, if
$\alpha, \beta \in \F,\,\beta \geq \alpha$.

 If $x\in\A$, there exist $\alpha\in\F$ such
that $x\in\A_\alpha$ and (a unique) $x_\beta\in\BB_\beta$ such
that $x=\phi_\beta(x_\beta)$, for all $\beta\geq\alpha$.\\ Then, we put
$$
j_{\beta\alpha}(x_\alpha):=x_\beta \mbox{\qquad if }\alpha\leq\beta.
$$

By (I.4), it follows easily that $j_{\beta \alpha}$  preserves the involution; i.e., $j_{\beta \alpha}(x_\alpha^*)=(j_{\beta \alpha}(x_\alpha))^*$.

The family $\{\BB_\alpha, j_{\beta\alpha}, \beta\geq \alpha\}$ is a {\em directed system of C*-algebras}, in the sense that:

\begin{itemize}\item[(J.1)]for every $\alpha, \beta \in \F$, with $\beta \geq \alpha$,
$j_{\beta \alpha}:\BB_\alpha \to \BB_\beta$ is a linear and
injective map; $j_{\alpha \alpha}$ is the identity of $\BB_\alpha$,
\item[(J.2)] for every $\alpha, \beta \in
\F$, with $\alpha \leq \beta$, \ $\phi_{\alpha} =\phi_{\beta}j_{\beta
\alpha}.$
 \item[(J.3)] $ j_{\gamma\beta}j_{\beta \alpha}=j_{\gamma \alpha}$, $ \alpha \leq \beta \leq \gamma$.
   \end{itemize}
   We assume that, in addition, the $j_{\beta \alpha}$'s are Schwarz maps (see, e.g. \cite{palmer}); i.e.,
\begin{itemize}  \item[(\textsf{sch})] $j_{\beta \alpha}(x_\alpha)^* j_{\beta \alpha}(x_\alpha)\leq j_{\beta \alpha}(x_\alpha^*x_\alpha), \quad \forall x_\alpha \in \BB_\alpha,\, \alpha\leq \beta$.
  \end{itemize}
For every $\alpha, \beta \in \F$, with $\alpha\leq \beta$, $j_{\beta \alpha}$ is continuous \cite{palmer} and, moreover, $$\|j_{\beta \alpha}(x_\alpha)\|_\beta\leq \|x_\alpha\|_\alpha, \quad \forall x_\alpha \in \BB_\alpha.$$

\medskip An involution in $\A$ is defined as follows..
Let $x \in \A$. Then $x \in \A_\alpha$, for some $\alpha \in \F$, i.e., $x= \phi_\alpha (x_\alpha)$, for a unique $x_\alpha\in \BB_\alpha$. Put
$x^*:= \phi_\alpha (x_\alpha^*)$. Then if $\beta \geq \alpha$, we
have
$$\phi_\beta^{-1}(x^*)=\phi_\beta^{-1}(\phi_\alpha (x_\alpha^*))=j_{\beta \alpha}(x_\alpha^*)=(j_{\beta \alpha}(x_\alpha))^*=x_\beta^*.$$
It is easily seen that the map $x \mapsto x^*$ is an involution in
$\A$. Moreover, by the definition itself, it follows that every map
$\phi_\alpha$ {\em preserves the involution}; i.e.,
$\phi_\alpha(x_\alpha^*)=(\phi_\alpha(x_\alpha))^*$, for all
$x_\alpha \in \BB_\alpha,\, \alpha \in \F$.

 \bedefi \label{defn_23}A locally convex vector space $\A$, with involution $^*$, is called a {\em C*-inductive locally convex space} if \begin{itemize}
             \item[(i)]  there exists a
family \mbox{$\{\{\BB_\alpha, \phi_\alpha\}, \alpha\in \F\} $}, where $\F$ is a direct set and, for every $\alpha \in \F$, $\BB_\alpha$ is a
C*-algebra and $\phi_\alpha$ is a linear injective map of $\BB_\alpha$ into $\A$,
satisfying the above conditions (I.1) - (I.4) and (\textsf{sch}), with $\A_\alpha= \phi_\alpha(\BB_\alpha)$, $\alpha \in \F$;
             \item[(ii)] $\A$ is endowed
with the locally convex inductive topology $\itop$ generated by the
family \mbox{$\{\{\BB_\alpha, \phi_\alpha\}, \alpha\in \F\} $}.
           \end{itemize}
\findefi
The family \mbox{$\{\{\BB_\alpha, \phi_\alpha\}, \alpha\in \F\} $} is called {\em the  defining system of $\A$}. We notice that the involution is automatically continuous in $\A[\itop]$.

A C*-inductive locally convex space has a natural positive cone.

An element $x\in \A$ is called \emph{positive} if there
exists $\gamma\in\F$ such that $\phi_\alpha^{-1}(x)\in\BB_\alpha^+$,
$\forall\alpha\geq\gamma$. \\
We denote by $\A^+$ the set of all positive elements of $\A$.

\medskip

Then,
\begin{itemize}
\item[(i)] Every positive element $x\in \A$ is hermitian; i.e., $x\in\A_h:=\{y\in\A:\,y^*=y\}$
.
\item[(ii)]$\A^+$ is a non empty convex pointed cone; i.e. $\A^+ \cap (-\A^+)=\{0\}$.
\item[(iii)] If $\alpha \in \F$ and $x_\alpha \in \BB_\alpha^+$, $\phi_\alpha (x_\alpha)$ is positive.
\end{itemize}

Moreover,
every hermitian element $x=x^*$ is the
difference of two positive elements, i.e. there exist
$x^+,x^-\in\A^+$ such that $x=x^+-x^-$.

A linear functional $\omega$ is said to be {\em positive} if $\omega(x)\geq 0$ for every $x=(x_\alpha)\in \A$.
As shown in \cite[Prop. 3.9, 3.10]{betra2}, $\omega$ is positive if, and only if, $\omega_\alpha (x_\alpha):= \omega(\phi_\alpha (x_\alpha))\geq 0$ for every $\alpha\in \F$. We write, in this case, $\omega=\varinjlim \omega_\alpha$.
\bigskip

\subsection{Bounded elements}

\bedefi Let $\A$ be a C*-inductive locally convex space. An element $x =(x_\alpha)\in \A$, with $x_\alpha \in \BB_\alpha$, is called {\em bounded} if $x\in \A_\alpha$, for every $\alpha \in \F$ and
$ \sup_{\alpha \in \F} \|x_\alpha\|_\alpha <\infty$.
The set of bounded elements of $\A$ is denoted by $\A_b$.
\findefi

\begin{prop} The set $\A_b$ is a Banach space under the norm $\|x\|_b=\sup_{\alpha \in \F} \|x_\alpha\|_\alpha$.
\end{prop}
\begin{proof} We only prove the completeness. Let $\{x_n\}$ be a Cauchy sequence in $\A_b$. Then for every $\alpha \in \F$ the sequence $\{x_n^\alpha\}$, with $x_n^\alpha:=(x_n)_\alpha$, is Cauchy in $\BB_\alpha$ so it converges to some $x_\alpha\in \BB_\alpha$. Since the $j_{\beta\alpha}$'s are continuous, one easily proves that the family $\{x_\alpha\}$ defines an element $x=(x_\alpha)$ of $\A$. From the Cauchy condition, for every $\epsilon >0$, there exists $n_\epsilon\in {\mb N}$ such that
\begin{equation} \label{eqn_CC}\sup_{\alpha \in \F} \|x_n^\alpha - x_m^\alpha\|_\alpha <\epsilon \end{equation}

If $m>n_\epsilon$,
$$ \|x_\alpha\|_\alpha \leq \|x_\alpha - x_m^\alpha\|_\alpha +\|x_m^\alpha\|_\alpha\leq\epsilon + \|x_m^\alpha\|_\alpha.$$
Hence,
$$\sup_{\alpha \in \F} \|x_\alpha\|_\alpha \leq \epsilon + \sup_{\alpha \in \F}\|x_m^\alpha\|_\alpha<\infty.$$
Thus $x \in \A_b$.

Fix now $n>n_\epsilon$ and let $m \to \infty$ in \eqref{eqn_CC}. Then,
$$\sup_{\alpha \in \F} \|x_n^\alpha - x_\alpha\|_\alpha \leq\epsilon .$$
This proves that $x_n \to x$.
\end{proof}

In what follows we will consider *-representations of a C*-inductive locally convex space. We recall the basic definitions.

\medskip

Let $\F$ be a set directed upward by $\leq$. A family $\{\H_\alpha, U_{\beta \alpha}, \alpha, \beta \in \F,
\beta\geq \alpha\}$,  where each
$\H_\alpha$ is a Hilbert space
(with inner product
$\ip{\cdot}{\cdot}_\alpha$ and norm $\|\cdot\|_\alpha$)  and, for every $\alpha, \beta \in \F$, with $\beta\geq \alpha$, $U_{\beta \alpha}$ is a linear map  from $\H_\alpha$ into  $\H_\beta$, is called a {\em directed contractive system of
Hilbert spaces} if the following conditions are satisfied
\begin{itemize}
                \item[(i)] $U_{\beta \alpha}$ is injective;
                \item[(ii)] $\|U_{\beta \alpha}\xi_\alpha\|_\beta \leq \|\xi_\alpha\|_\alpha, \quad \forall \xi_\alpha \in \H_\alpha$;
                \item[(iii)] $U_{\alpha\alpha}=I_\alpha$, the identity of $\H_\alpha$;
                \item[(iv)] $U_{\gamma\alpha}=U_{\gamma\beta}U_{\beta\alpha}$, $\alpha\leq \beta\leq \gamma$.
              \end{itemize}

A directed contractive system of
Hilbert spaces defines a conjugate dual pair $(\D^\times,\D)$ which is called
the {\em joint topological limit} \cite{betra} of the directed contractive system
$\{\H_\alpha, U_{\beta \alpha}, \alpha, \beta \in \F, \beta\geq
\alpha\}$ of Hilbert spaces.

\bedefi \label{prop_induclimrepres}
Let $\A$ be the C*-inductive
locally convex space defined by the system
\mbox{$\{\{\BB_\alpha, \Phi_\alpha\}, \alpha\in \F\} $} as in
Definition \ref{defn_23}.

For each $\alpha \in \F$, let  $\pi_\alpha$ be a *-representation
of $\BB_\alpha$ in Hilbert space $\H_\alpha$. The collection $\pi:=\{\pi_\alpha\}$ is said to be a *-representation of $\A$ if
\begin{itemize}
\item[(i)] for every $\alpha, \beta \in \F$ there exists a linear map $U_{\beta\alpha}: \H_\alpha \to \H_\beta$ such that the family $\{\H_\alpha, U_{\beta \alpha}, \alpha, \beta \in \F, \beta\geq
\alpha\}$ is a directed contractive system of
Hilbert spaces;
\item[(ii)] the following equality holds
\begin{equation}\label{eq_coherence} \pi_\beta (j_{\beta\alpha} (x_\alpha)) = U_{\beta\alpha}\pi_\alpha(x_\alpha) U_{\beta\alpha}^*, \quad \forall x_\alpha \in \BB_\alpha, \, \beta \geq \alpha.\end{equation}
\end{itemize}
In this case we write $\pi(x)= \varinjlim \pi_\alpha (x_\alpha)$ for every $ x=(x_\alpha)\in \A$ or, for short, $\pi= \varinjlim \pi_\alpha$.

The *-representation $\pi$ is said to be {\em faithful} if $x \in \A^+$ and $\pi(x)=0$ imply $x=0$ (of course, $\pi(x)=0$ means that there exists $\gamma\in \F$ such that $\pi_\alpha (x_\alpha)=0$, for $\alpha\geq \gamma$).
\findefi
\berem With this definition (which is formally different from that given in \cite{betra2} but fully equivalent), $\pi(x)$, $x\in \A$, is not an operator but rather a collection of operators.
But as shown in \cite{betra2}, $\pi(x)$ can be regarded as an operator acting on the joint topological limit $(\D^\times, \D)$ of $\{\H_\alpha, U_{\beta \alpha}, \alpha, \beta \in \F, \beta\geq
\alpha\}$. The corresponding space of operators was denoted by $\LBDDs$; it behaves in the very same way as the space $\LBDD{}$ studied in Section \ref{sect_oprhs} and reduces to it when the family of Hilbert spaces is exactly $\{ \H_A; A \in \LD\}$. The main difference consists in the fact that the $\H_\alpha$'s need not be all subspaces of a certain Hilbert space $\H$.
\enrem
\begin{lemma} \label{lemma45}Let $\pi=\varinjlim \pi_\alpha$ be a faithful *-representation of $\A$. Then, for every $\alpha \in \F$, $\pi_\alpha$ is a faithful *-representation of $\BB_\alpha$.

\end{lemma}
\begin{proof}
Let $x_\alpha\in \BB_\alpha^+$ with $\pi_\alpha(x_\alpha)=0$. Let $x\in \A$ be the unique element of $\A$ such that $x=\phi_\alpha(x_\alpha)$.
Then $\pi_\beta(x_\beta)=\pi_\beta(j_{\beta\alpha}(x_\alpha))= U_{\beta\alpha} \pi_\alpha(x_\alpha)U_{\beta\alpha}^*=0$. Hence $\pi(x)=0$ and, therefore $x=0$. Thus there exists $\overline{\gamma}\in \F$ such that $x_\gamma=0$, for $\gamma\geq \overline{\gamma}$. Let $\beta\geq \alpha, \overline{\gamma}$. Then $0=x_\beta= j_{\beta\alpha}(x_\alpha)$. Hence, by the injectivity of $j_{\beta\alpha}$, $x_\alpha=0$.
\end{proof}

As shown in \cite[Proposition 3.16]{betra2}, if a C*-inductive locally convex space $\A$ fulfills the following conditions

\begin{itemize}
\item[({\sf r}$_1$)]if $x_\alpha \in \BB_\alpha$ and $j_{\beta\alpha}(x_\alpha)\geq 0$, $\beta\geq \alpha$, then $x_\alpha\geq 0$;
\item[({\sf r}$_2$)]$e_\beta \in j_{\beta\alpha}(\BB_\alpha), \quad \forall \alpha, \beta \in \F, \beta \geq \alpha$;
\item[({\sf r}$_3$)] every  positive linear functional $\omega=\varinjlim \omega_\alpha$ on $\A$ satisfies the following property
\begin{itemize}
\item[$\bullet$] if $\alpha \in \F$ and $\omega_\beta (j_{\beta\alpha}(x_\alpha^*)j_{\beta\alpha}(x_\alpha))=0$, for some $\beta > \alpha$ and $x_\alpha \in \BB_\alpha$, then $\omega_\alpha (x_\alpha^* x_\alpha)=0$;
\end{itemize}
\end{itemize}
then, $\A$ admits a faithful representation. The conditions ({\sf r}$_1$), ({\sf r}$_2$), in fact, guarantee that $\A$ possesses sufficiently many positive linear functionals, in the sense that for every $x \in \A^+$, $x\neq 0$ there exists a positive linear functional $\omega$ on $\A$ such that $\omega(x)>0$ \cite[Theorem 3.14]{betra2}.

\begin{thm} Let $\A$ be a C$^*$-inductive locally convex space and $x=(x_\alpha) \in \A$.  The following statements hold.
\begin{itemize}
\item[(i)] If $x\in \A_b$, then, for every *-representation $\pi=\varinjlim \pi_\alpha$ of $\A$, one has $$\sup_{\alpha \in \F}\|\pi_\alpha (x_\alpha)\|_{\alpha\alpha}< \infty,$$
where $\|\cdot \|_{\alpha\alpha}$ denote the norm of $\BB(\H_\alpha)$.

\item[(ii)]Conversely, if $\A$
admits a faithful *-representation $\pi^f=\varinjlim \pi_\alpha^f$ and $$\sup_{\alpha \in \F}\|\pi_\alpha^f (x_\alpha)\|_{\alpha\alpha}<\infty,$$ then  $x\in\A_b$.
\end{itemize}
\end{thm}
\begin{proof} (i):\; For every $\alpha \in \F$, $\pi_\alpha$ is a *-representation of the C*-algebra $\BB_\alpha$. Hence
$$\|\pi_\alpha (x_\alpha)\|_{\alpha\alpha}\leq \|x_\alpha\|_\alpha.$$ Thus if $x \in \A_b$ the statement follows immediately from the definition.

\noindent(ii):\; Let  $ \pi^f(x)= \varinjlim \pi_\alpha^f (x_\alpha)$. Then, by Lemma \ref{lemma45}, for every $\alpha \in \F$, $ \pi_\alpha^f$ is a faithful re\-pre\-sen\-tation of $\BB_\alpha$. The *-re\-pre\-sen\-tation $\pi_\alpha^f$ is an isometric isomorphism of C$^*$-algebras, for all $\alpha\in\F$; hence  $$\sup_{\alpha \in \F}\|x_\alpha\|_\alpha=\sup_{\alpha \in \F}\|\pi_\alpha^f(x_\alpha)\|_{\alpha\alpha}<\infty.$$ This proves that $x$ is a bounded element of $\A$.
\end{proof}

\subsection{Order bounded elements}
Let $\A$ be a C*-inductive locally convex space. If $x \in \A$, we put
$$ \Re(x) = \frac{x+x^*}{2} \; \mbox{ and } \; \Im(x)=\frac{x-x^*}{2i}.$$ Both $\Re(x)$ and $\Im(x)$ are symmetric elements of $\A$.

Assume that $\A$ has an element $u = u^{*}$ such that $\|u_\alpha \|_\alpha \leq 1$, for every $\alpha \in \F$, and there exists $\gamma\in \F$ such that $u_\beta = j_{\beta \gamma } (e_\gamma)$ $\forall \beta \geq \gamma$, ($e_\gamma$ is the unit of $\mathfrak{B}_\gamma$). For shortness we call the element $u$ a {\em pre-unit} of $\A$.

\berem
The pre-unit $u\in\mathfrak{A}$, if any, is unique. Indeed, let suppose there is another $v\in\mathfrak{A}$ satisfying the same properties as $u$. Then, $$\exists \gamma , \gamma ' \in \F;\; u_\beta = j_{\beta \gamma } (e_\gamma) ,\; v_{\beta'}  = j_{\beta ' \gamma ' } (e_\gamma ') ,\quad \forall \beta \geq \gamma , \beta ' \geq \gamma '$$
so, if $\delta\geq\gamma , \gamma '$, one has $u_\lambda = v_\lambda $, $\forall \lambda\geq\delta$.
\enrem

\bedefi Let $\A$ be a C*-inductive locally convex space with pre-unit $u$. We say that $x \in \A$ is {\em order bounded} (with respect to $u$) if there exists $\lambda >0$ such that
$$ -\lambda u \leq \Re(x) \leq \lambda u \qquad -\lambda u \leq \Im(x) \leq \lambda u.$$
\findefi

{\begin{thm} \label{bounded_order}
Let $\A$ be a C*-inductive locally convex space satisfying condition $\mbox{\rm({\sf r}$_1$)}$. Assume that $\A$ has a pre-unit $u$.

 Then, $x \in \A_b$ if, and only if, $x$ has a representative for every $\alpha \in \F$ (i.e. for every $\alpha \in \F$, there exists $x_\alpha \in \BB_\alpha$ such that $x=\phi_\alpha (x_\alpha)$)  and $x$ is order bounded with respect $u$.
\end{thm}
\begin{proof} Let us assume that $ x =x^* \in \A_b $. Then, $x$ has a representative $x_\alpha$, with $x_\alpha^* = x_\alpha$, in every $\BB(\H_\alpha)$ and $\lambda:= \sup_{\alpha\in \F}\|x_\alpha\|_\alpha <\infty.$ Hence, we have $$ -\lambda e_\alpha \leq x_\alpha \leq \lambda e_\alpha , \quad \forall \alpha \in \F,$$ where $e_\alpha$ denotes the unit of $\BB_\alpha$. By the definition of $u$, there exists $\gamma\in \F$ such that $u_\beta= j_{\beta\gamma}(e_\gamma)$ for $\beta\geq \gamma$. Hence, taking into account that the maps $j_{\beta \alpha}$ preserve the order, we have $$ -\lambda u_\beta \leq x_\beta \leq \lambda u_\beta , \quad \forall \beta\geq\gamma.$$ This implies that $-\lambda u\leq x\leq \lambda u$.\\
Now, let us suppose that for some $\lambda>0$,  $-\lambda u\leq x\leq \lambda u$. Then, there exists $\gamma\in \F $ such that \begin{equation}\label{eq_order} -\lambda u_\beta \leq x_\beta \leq \lambda u_\beta ,\qquad \forall \beta\geq\gamma .\end{equation} Let now $\alpha \in \F$. Then, there is $\delta\geq \alpha, \gamma$ such that \eqref{eq_order} holds for $\delta\geq \alpha$.
Hence, by using ({\sf r}$_1$), we conclude that $$ -\lambda u_\alpha \leq x_\alpha \leq \lambda u_\alpha\quad \forall \alpha\in \F .$$  This implies that, $\|x_\alpha\|_\alpha\leq \lambda$, for every $\alpha \in \F$. Thus, $ x\in \mathfrak{A}_b.$
\end{proof}

From the proof of the previous theorem it follows easily that
\begin{prop}\label{4.11}
Let $x=x^*\in\mathfrak{A}_b$ and put $$p(x) = \inf \{\lambda >0; \, -\lambda u\leq x\leq \lambda u \}.$$ Then, $p(x) = \|x\|_b$.
\end{prop}

\section{C*-inductive partial *-algebras}
As shown in \cite{betra2}, a partial multiplication in $\A$ can be defined by a family $w=\{w_\alpha\}$, $w_\alpha \in \BB_\alpha$.
Let $w=\{w_\alpha\}$ be a family of elements, such that each
$w_\alpha \in\BB_\alpha^+$ and $j_{\beta \alpha}(w_\alpha)=w_\beta$,
for all $\alpha, \beta \in \F$ with $\beta \geq \alpha$.

Let $x,y\in\A$. The partial multiplication $x\cdot y$ 
is defined by the conditions:
\begin{align*}\label{_defn_multiplication}
&\exists \gamma\in
\F:\,\phi_\beta(\phi^{-1}_\beta(x)w_\beta\phi^{-1}_\beta(y)) =
\phi_{\beta'}(\phi^{-1}_{\beta'}(x)w_{\beta'}\phi^{-1}_{\beta'}(y)), \; \forall\beta,\beta'\geq\gamma \\
& x\cdot y=\phi_\beta(\phi^{-1}_\beta(x)w_\beta \phi^{-1}_\beta(y)),
\quad \beta\geq\gamma.
\end{align*}

Then, $\A$ is an {\em associative} partial *-algebra with respect to the usual operations and
the above defined multiplication (see \cite[Sect. 2.1.1]{ait_book} for the definitions) and we will call it a {\em C*-inductive partial *-algebra}.

The partial *-algebra $\A$ has a unit $e$ (that is, an element $e$ which is a  left- and right universal multiplier such that $x\cdot e=e\cdot x=x$, for every $x \in \A$)  if, and only if, every element $w_\alpha$ of the family $\{w_\alpha\}$ defining the multiplication is invertible and
\begin{equation}\label{eqn_unit}  j_{\beta\alpha}(w_\alpha^{-1})=w_\beta^{-1}, \quad \forall \alpha, \beta\in\F, \,\beta\geq\alpha.
\end{equation}
In this case, $e=\phi_\alpha(w_\alpha^{-1})$, independently of
$\alpha\in \F$.

The element $e$ is called a {\em bounded unit} if it is a bounded element of $\A$ and $\|e\|_b=1$.

\begin{prop} Let $\A$ be a C*-inductive partial *-algebra with the multiplication defined by a family $\{w_\alpha\}$. Assume that  $e=(w_\alpha^{-1})$ is a bounded unit of $\A$. Then $\A_b$ is a Banach partial *-algebra; that is, $\A_b[\|\cdot\|_b] $ is a Banach space with isometric involution $^*$ and there exists $C\geq 1$ such that the following inequality holds
\begin{equation}\label{ineq_C} \|x\cdot y\|_b \leq C\|x\|_b\|y\|_b, \quad \forall x,y \in \A_b \mbox{\; with  $x\cdot y$ well-defined}.\end{equation}
\end{prop}

\berem The constant $C$ in \eqref{ineq_C} can be taken equal to $1$ if $w_\alpha^{-1}=e_\alpha$, for each $\alpha \in \F$, where $e_\alpha$ is the unit of the C*-algebra $\BB _\alpha$.
Under the same assumption, the norm of $\A_b$ satisfies the C*-property, which in our case reads
$$ \|x^*\cdot x\|_b = \|x\|_b^2, \quad \forall x \in \A_b \mbox{ with  $x^*\cdot x$ well-defined}.$$
This is no longer true in the general case.
\enrem

\berem In Example 5.3 of \cite{betra2} two of us tried to construct a family $\{W_A\in \BB(\Hil_A);\, A\in \LD \}$ so that the partial multiplication defined in $\LBDD{}$ by the method mentioned above would reproduce the quasi *-algebra structure of $(\LBDD{},\LD)$ (see Section \ref{sect_preliminaries}). Unfortunately, the conclusion of that discussion is uncorrect (see, \cite{bt_erratum} for more details).
\enrem

Let $\A$ be a C*-inductive partial *-algebra with the multiplication defined by a family $\{w_\alpha\}$ as above.
The spaces $R\A$   and $L\A$ of the right-, respectively, left universal multipliers (with respect to $w$) of $\A$ are algebras. Hence, $\Ao:=L\A\cap R\A$ is
a *-algebra and, thus,
\begin{itemize}
  \item[(i)] $(\A,\Ao)$ is a quasi *-algebra.
  \item[(ii)]If $\A$ is endowed with $\itop$, then the maps $x  \mapsto x^*$, $x\mapsto a\cdot x$, $x \mapsto x\cdot b$, $a,b \in \Ao$ are continuous.
\end{itemize}

\medskip
It is easily seen from the very definition that, if $a \in R\A$ and $x \in \A^+$, then $a^*xa \in \A^+$. Hence, if ${\mc P}(\mathfrak{A})$ denotes the family of all positive linear functionals on $\mathfrak{A}$, we have in particular $\omega(a^*xa)\geq 0$, for every $\omega \in {\mc P}(\A)$.

\begin{thm}\label{thm_end}
Let $\A$ be a C*-inductive partial *-algebra with the multiplication defined by a family $\{w_\alpha\}$ and with pre-unit $u$. Assume, moreover, that the following condition $\mbox{\rm({\sf P})}$ holds:
\begin{itemize}
\item[({\sf P})] $y\in \mathfrak{A}$, $\omega(a^*ya)\geq 0$, $\forall \omega \in {\mc P}(\mathfrak{A}) $ and $a\in R\mathfrak{A}$ $\Rightarrow$ $y\in \mathfrak{A}^{+}$ ;\end{itemize}
then, for $x\in \mathfrak{A}$, the following conditions are equivalent.
\begin{itemize}
\item[(i)] $x$ is order bounded with respect to $u$.\\
\item[(ii)]There exists $\gamma_x >0$ such that $$|\omega(a^*xa)|\leq \gamma_x \omega (a^*ua) , \quad\forall \omega \in {\mc P}(\mathfrak{A}), \quad \forall a \in R\A. $$
\item[(iii)] There exists $\gamma_x >0$ such that $$|\omega(b^*xa)|^2\leq\gamma_x\omega(a^*ua)\omega(b^*ub), \quad\forall \omega \in {\mc P}(\mathfrak{A}),\;\forall a,b\in R\A.$$
\end{itemize}
\end{thm}
\begin{proof}
It is sufficient to consider the case $x=x^*$;\\
(i)$\Rightarrow $(ii): Let $\omega \in {\mc P}(\mathfrak{A})$. By the hypothesis, $-\gamma u\leq x\leq\gamma u$, for some $\gamma>0$; then $\omega(\gamma u - x)\geq 0$ and $\omega(a^* (\gamma u - x) a )\geq 0$, $\forall a \in R\A$. On the other hand, similarly, one can show that  $\omega(a^* (x- \gamma u ) a )\geq 0$.\\
(ii)$\Rightarrow $(i): Assume now that $u$ is a pre-unit and there exists $\gamma_x >0$ such that $$|\omega(a^*xa)|\leq \gamma_x \omega (a^*ua) ,\quad \forall \omega \in {\mc P}(\mathfrak{A}), \quad a \in R\A .$$ Then $$\gamma_x \omega (a^*ua)\pm \omega(a^*xa)\geq 0 \Rightarrow\omega(a^*(\gamma_x u\pm x)a)\geq 0,\quad \forall \omega \in {\mc P}(\mathfrak{A}), a \in R\A.$$
So, by ({\sf P}), $\gamma_x u\pm x\geq 0.$\\
(i)$\Rightarrow$(iii): By the assumption,
there exists $\gamma>0$ such that $-\gamma u\leq x\leq\gamma u$. Let $ \omega \in {\mc P}(\mathfrak{A})$. Then,
the linear functional $\omega_a$ on $\A$, defined by $\omega_a(x):=\omega(a^* x a)$, is positive. Hence, if $x=x^*$
$$  -\gamma \omega_a( u)\leq \omega_a(x)\leq\gamma \omega_a(u);$$
i.e.,
$$|\omega(a^*xa)|\leq \gamma \omega (a^*ua).$$

Now, let $x\in\A^+$, $a,b\in R\A$. Let us define $\Omega^x_\omega(a,b):=\omega(b^*  x  a)$.
Then, it is easily checked that $\Omega^x_\omega$ is a positive sesquilinear form on
$R\A\times R\A$. Using the
the Cauchy-Schwartz inequality we obtain
\begin{eqnarray*}
\label{ineq_second} |\omega(b^*  x  a)|&\leq & \omega(a^*  x  a)^{1/2}\omega(b^*  x  b)^{1/2} \\
   &\leq& \gamma
   \omega(a^* u  a)^{1/2}\omega(b^* u  b)^{1/2}.
\end{eqnarray*}
The extension to arbitrary $x \in \A$ goes through as in the proof of Proposition 4.3 of \cite{betra2}.

(iii)$\Rightarrow$(ii) It is trivial.
\end{proof}

The previous proof shows that if $x=x^*\in \A$ is order bounded with respect to $u$ then
$$p(x) \leq \sup\{ |\omega(b^*  x  a)|; \omega\in {\mc P}(\mathfrak{A}); \, a,b\in R\A; \omega(a^* u  a)=\omega(b^* u  b)=1\}.$$
where $p(x)$ is the quantity defined in Proposition \ref{4.11}.

\medskip

The following statement is an easy consequence of Theorem \ref{4.11} and Theorem \ref{thm_end}.
\begin{thm}Let $\A$ be a C*-inductive partial *-algebra with the multiplication defined by a family $\{w_\alpha\}$ and pre-unit $u$. Assume that conditions $\mbox{\rm({\sf r}$_1$)}$ and ({\sf P}) are satisfied. For an element $x \in \A$, having a representative in every $\BB_\alpha$, $\alpha \in \F$, the following statements are equivalent.
\begin{itemize}
\item[(i)] $x\in \A_b$.
\item[(ii)] $x$ is order bounded with respect to $u$.
\item[(iii)] For every $\omega \in {\mc P}(\mathfrak{A})$
$$|\omega(b^*xa)|^2\leq\gamma_x\omega(a^*ua)\omega(b^*ub), \quad\forall a,b\in R\A.$$
\end{itemize}
\end{thm}

\end{document}